\documentclass[10pt]{amsart}
\usepackage{latexsym, amsmath, amssymb}
\usepackage{color}
\usepackage{epsfig,graphics}

\newcommand{\newsection}[1]
{\section{#1}\setcounter{theorem}{0} \setcounter{equation}{0}
\par\noindent}

\newtheorem{theorem}{Theorem}

\newtheorem{lemma}[theorem]{Lemma}

\newcommand{\cd}{\, \cdot\, }
\newcommand{\supp}{\text{supp }}
\newcommand{\R}{{\mathbb R}}
\newcommand{\Z}{{\mathbb Z}}

\newcommand{\ang}{{\not\negmedspace\nabla}}

\newcommand{\la}{\langle}
\newcommand{\ra}{\rangle}

\newcommand{\e}{\epsilon}

\newcommand{\M}{{\mathcal{M}}}

\newcommand{\tv}{{\tilde{v}}}

\newcommand{\weight}{{\Bigl(1-\Bigl(\frac{r_s}{r}\Bigr)^{d+1}\Bigr)}}
\renewcommand{\S}{{\mathbb{S}}}

\begin{document}

\title[Localized energy estimates on Myers-Perry space-times]
{
Localized energy estimates for wave equations on $(1+4)$-dimensional
Myers-Perry space-times
}

\thanks{
The second author was supported in part by NSF grant DMS-1054289}

\author{Parul Laul}
\author{Jason Metcalfe}
\author{Shreyas Tikare}
\author{Mihai Tohaneanu}

\address{Department of Mathematics and
  Computer Science, Bronx Community College, 2155 University Ave.,
  Bronx, NY  10453 }

\address{Department of Mathematics, University of North Carolina,
  Chapel Hill, NC 27599-3250}

\address{Department of Mathematics, University of North Carolina,
  Chapel Hill, NC 27599-3250}

\address{Johns Hopkins University, Department of Mathematics,
  Baltimore, MD  21218}

\begin{abstract} Localized energy estimates for the wave equation have
  been increasingly used to prove various other dispersive estimates.
  This article focuses on proving such localized energy estimates on
  $(1+4)$-dimensional Myers-Perry black hole backgrounds with small
  angular momenta.  The
  Myers-Perry space-times are generalizations of higher dimensional
  Kerr backgrounds where additional planes of rotation are availabile
  while still maintaining axial symmetry.  Once it is determined that
  all trapped geodesics have constant $r$, the method developed by
  Tataru and the fourth author, which perturbs off of the
  Schwarzschild case by using a pseudodifferential multiplier, can be adapted.
\end{abstract}

\maketitle

%%%%%%%%%%%%%%%%%%%%%%%%%%%%%%%%%%%%%%%%%%%%%%%%%%%%%%%%%%%%%%%%%%%%%%%%%%%%%%%%%%%%%%%%%
%%%%%%%%%%%%%%%%%%%%%%%%%%%%%%%%%%%%%%%%%%%%%%%%%%%%%%%%%%%%%%%%%%%%%%%%%%%%%%%%%%%%%%%%%
%%%%%%%%%%%%%%%%%%%%%%%%%%%%%%%%%%%%%%%%%%%%%%%%%%%%%%%%%%%%%%%%%%%%%%%%%%%%%%%%%%%%%%%%%

\newsection{Introduction}
This article focuses on proving localized energy estimates for wave
equations on the family of $(1+4)$-dimensional Myers-Perry spacetimes
with small angular momenta.  The latter are higher dimensional
generalizations of the Kerr family of solutions to the Einstein vacuum
equations.  They are axisymmetric, though with the additional
dimension, an extra nonzero angular momentum may be permitted, which
further complicates the behavior of the null geodesics.  On, e.g., the
Schwarzschild and Kerr space-times, such
localized energy estimates have been essential to proving other types
of dispersive estimates such as Strichartz estimates \cite{MMTT}, \cite{Tohaneanu} and
pointwise decay estimates \cite{MTT}, \cite{Tataru}, \cite{DR2, DR3},
\cite{Luk1, Luk2}.

We first describe the localized energy estimates on the Minkowski
space-time.  To begin, we define
\[\|u\|_{LE_M} = \sup_{j\in \Z} 2^{-\frac{j}{2}} \|u\|_{L^2_{t,x}(\R
  \times \{|x|\in [2^{j-1},2^j]\})}\]
and
\[\|u\|_{LE^1_M} = \|u'\|_{LE_M} + \||x|^{-3/2} u\|_{L^2_{t,x}},\]
where $u' = (\partial_t u,\nabla_x u)$ is the space-time gradient.  To
measure an inhomogeneous term, we shall use a dual norm
\[\|f\|_{LE^*_M} = \sum_{k\in \Z}
2^{\frac{k}{2}}\|f\|_{L^2_{t,x}(\R\times \{|x|\in [2^{k-1},2^k]\})}.\]
Then for $\Box=-\partial_t^2+\Delta$, the localized energy estimate
for the wave equation states that
\begin{equation}
  \label{LEMink}
  \|u'\|_{L^\infty_t L^2_x} + \|u\|_{LE^1_M} \lesssim
  \|u'(0,\cd)\|_{L^2} + \|\Box u\|_{LE^*_M+ L^1_tL^2_x}
\end{equation}
provided $n\ge 4$.  A similar estimate holds when $n=3$ but the second
term in the definition of the $LE^1_M$ norm must be replaced by, e.g.,
$\||x|^{-1} u\|_{LE_M}$.  Estimates of this type first appeared in \cite{M}.
From this original estimate, it is clear that the $l^\infty_j$ norm in
the definition of $\|\,\cdot\,\|_{LE_M^1}$
can be replaced by square summability, which is stronger, on the angular
derivative term and, when $n\ge 4$, on the lower order term.
Generalizations of the original estimate have appeared in, e.g.,
\cite{Strauss}, \cite{KPV}, \cite{SmSo}, \cite{KSS}, \cite{Met}, \cite{BPSTz}
\cite{HY}, 
\cite{St}, \cite{MS, MS2}, \cite{MT, MT2} and have ultimately led to \eqref{LEMink}.  These estimates have
been applied to problems in scattering theory, used to prove long time
existence for nonlinear
wave equations, and used to prove other dispersive estimates
by handling the behavior of the solution in a compact region.  This
class of estimate is known to be fairly robust and similar estimates
are known for small, possibly time dependent, long-range perturbations
of the Minkowski metric \cite{Alinhac}, \cite{MS, MS2}, \cite{MT, MT2} and for time independent,
nontrapping, asymptotically flat perturbations \cite{BH}, \cite{Burq}, \cite{SW}.

While proofs that rely on Huygens' principle in odd spatial
dimensions (see, e.g., \cite{KSS}), Plancherel's theorem (see, e.g., \cite{SmSo}, \cite{Met}), and
resolvant estimates (see, e.g., \cite{BH}, \cite{Burq}) exist, the most robust argument and the
one that is most applicable to the current setting relies, in essence,
on a positive commutator argument.  See, e.g., \cite{St}, \cite{MS}.  For a choice of function $f(r)$, one multiplies the equation
$\Box u$ by $f(r)\partial_r u + \frac{n-1}{2}\frac{f(r)}{r} u$ and
integrates by parts.  The original estimate of \cite{M}, then, corresponds
to the choice $f(r)\equiv 1$.

Related procedures have been carried out on the $(1+3)$-dimensional
Schwarzschild space-times, beginning with \cite{LS}.  The Schwarzschild
space-times are the class of spherically symmetric, static solutions
to Einstein's equations and represent the simplest black hole
solutions.  When the analysis is carried out here, the phenomenon of
trapping is encountered.  The Schwarzschild space-time contains null
geodesics that stay in a compact region for all times.  Specifically
trapping occurs on the event horizon $r=2M$ and on the photon sphere
$r=3M$.

Trapping is a known barrier to localized energy estimates.  See, e.g.,
\cite{Ralston} and \cite{Doi} (the latter regards the related local smoothing estimate
for the Schr\"odinger equation).  And, thus, estimates on
Schwarzschild are expected to contain losses when compared to
\eqref{LEMink}; recently this was shown rigurously in \cite{Sb}.  By modifying the choice of $f(r)$ and, in
particular, requiring vanishing on the trapped sets, localized energy
estimates (with losses) were obtained for the wave equation on
Schwarzschild space-times in \cite{BS1, BSerrata, BS3, BS5}, \cite{DR, DR2}, \cite{MMTT}.

While trapping occurs on the event horizon $r=2M$, it was discovered \cite{MMTT}
that, by quantifying the red-shift effect in the style of \cite{DR}, 
this loss in the localized energy estimates can be negated.  On the
other hand, the
loss at the photon sphere typically arises as a quadratic vanishing of
the coefficient on the angular derivatives and the time derivative,
though with some, e.g., microlocal analysis this can be improved to
merely a logarithmic loss.  See \cite{MMTT} and the preceding works
\cite{BS5, BS3}.  

The Kerr family of solutions to Einstein's equations correspond to
axially symmetric, rotating black holes.  For small angular momentums,
the Kerr metrics are small short-range perturbations of the
Schwarzschild metric.  The structure of the trapping, however, is much
more complicated.  While still of codimension two in phase space, the
location of the trapping may no longer be described merely in physical
space.  As such, it is provably impossible \cite{Alinhac2} to obtain such a
localized energy estimate using a first order differential multiplier.

Despite this, three approaches have been developed for proving such on Kerr backgrounds
with small angular momenta.  In \cite{AB}, \cite{DR5, DR6, DR7}, and \cite{TT} respectively, the
authors proved certain versions of localized energy estimates using somewhat different approaches: \cite{AB} uses the existence of a nontrivial Killing tensor
(due to \cite{Carter2}), \cite{DR5, DR6, DR7} rely on frequency decomposition, and \cite{TT} is based on using a pseudodifferential
multiplier.  The three approaches are intimately related \cite{Carter}.  The
current study, however, relies most heavily on the ideas from \cite{TT}.

Though we shall currently focus only on the small angular momenta
regime, we mention some recent related works \cite{DR7}, \cite{Ar1, Ar2, Ar3} that permit
large momenta.  We also mention here a few places \cite{MMTT},
\cite{Tohaneanu}, \cite{Tataru}, \cite{MTT}, \cite{LMSTW}, \cite{DR2,
  DR3}, \cite{Luk1, Luk2} where example
applications of these estimates on black hole backgrounds can be found.  In particular, in these works, the
localized energy estimates are used to prove other measures of
dispersion.  They are used to obtain information about solutions in a
compact region where the geometry is most difficult and are combined
with known Minkowski estimates near infinity where the metric may be
viewed as a small perturbation of the Minkowski metric.  Such
arguments are akin to those appearing much earlier in, e.g., \cite{JSS} and to those used in, e.g., \cite{SmSo}, \cite{Burq},
\cite{KSS}, \cite{Met, Met2}, \cite{MS, MS2}, \cite{MT2}  to prove exterior domain analogs of
well-known boundaryless estimates.

Higher dimensional black hole space-times have been derived, based partially in interest coming from string theory.  Our motivation for studying this problem
comes from the connections to the problem of black hole stability,
which is described e.g. in \cite{DR5}.  It is well-known that classes of
quasilinear wave equations have global small data solutions in
$(1+n)$-dimensions for $n\ge 4$, but in order to guarantee such for
$n=3$, one must take advantage of some nonlinear structure.  See,
e.g., \cite{Sogge}.  As such,
one might consider such stability problems in higher dimensions first
where the nonlinear structure may be less essential.  The reader can
compare, e.g., the proof of the stability of Minkowski space \cite{CK} with
that in higher dimensions \cite{CCL}.

While the extra dimensions permit black holes with different
topologies (see e.g. \cite{ER2, ER1}, \cite{GP}, \cite{GS}, \cite{GL},
\cite{Horo},
\cite{Kunz}, \cite{Reall}, and the references therein), we shall study only those which are most closely
related to the standard Kerr solutions.  In particular, a higher
dimensional analog of the Schwarzschild space times was derived in
\cite{T}, and the analogs of Kerr are from \cite{MP}.  In higher dimensions, as
there are more directions in which one may rotate and yet maintain
axial symmetry, a broader family of solutions, called Myers-Perry
space-times, may be examined.

Localized energy estimates were proved independently on the (hyperspherical)
$(1+n)$-dimensional Schwarzschild space-times of \cite{T} in \cite{LM} and
\cite{Schlue}.  Rather than examining generic dimension, we examine
specifically $(1+4)$-dimensions.  In particular, as in \cite{TT}, our proof
will rely heavily on the integrability of the geodesic flow.  This is known
for Myers-Perry black holes with two distinct angular momenta
\cite{Palmer}, \cite{VSP},
but, to our knowledge, not generically.  In $(1+4)$-dimensions,
however, this covers the entire family.  Our main theorem, Theorem \ref{Kerr},
states that for small angular momenta there is a localized energy
estimate for the wave equation on such $(1+4)$-dimensional Myers-Perry
solutions.

This paper is structured as follows.  
The next section contains a review of the localized
energy estimates on $(1+4)$-dimensional Schwarzschild space-times.  In
what follows, we shall perturb off of this result.  We modify the
existing estimate of \cite{LM} by incorporating the redshift effect
and by smoothing out the multiplier near the photon sphere, which will
simplify the microlocal analysis that comes later.
In the third section, we present
the Myers-Perry metric and its most relevant properties.  We, there,
analyze the trapped geodesics and, in particular, prove that they each
lie on surfaces of constant $r$.   In the fourth
section, we define our local energy spaces, and in the final section,
we prove the main estimate, Theorem \ref{Kerr}.

%%%%%%%%%%%%%%%%%%%%%%%%%%%%%%%%%%%%%%%%%%%%%%%%%%%%%%%%%%%%%%%%%%%%%%%%%%%%%%%%%%%
%%%%%%%%%%%%%%%%%%%%%%%%%%%%%%%%%%%%%%%%%%%%%%%%%%%%%%%%%%%%%%%%%%%%%%%%%%%%%%%%%%%
%%%%%%%%%%%%%%%%%%%%%%%%%%%%%%%%%%%%%%%%%%%%%%%%%%%%%%%%%%%%%%%%%%%%%%%%%%%%%%%%%%%
\newsection{Localized energy on $(1+4)$-dimensional Schwarzschild}\label{Schwarzschild}
Here we shall rely on the approach of \cite{MMTT} and \cite{LM}.  Estimates of a similar
form were also proved in \cite{Schlue}.  Akin to the relationship
between \cite{MMTT} and \cite{TT}, the estimate of \cite{LM} plays a key role
as we assume that $a,b\ll 1$ and argue perturbatively.

In the sequel, we shall be employing pseudodifferential multipliers.
So we first seek to modify the multiplier of \cite{LM}, which is only
$C^2$ at the photon sphere, to make it smooth in a neighborhood of the
photon sphere.  Moreover, as in \cite{MMTT}, we shall include the
arguments of \cite{DR} that take advantage of the red-shift
effect and allow us to prove an estimate that does not degenerate at
the event horizon.  This was previously done in \cite{Schlue}, and we
include it here for completeness.

We assume a basic familiarity with \cite{LM} and shall not reproduce
every calculation here.

The metric for the $(1+n)$-dimensional hyperspherical Schwarzschild
space-time, which was discovered in \cite{T}, is
\[ds^2 = -\weight dt^2 + \weight^{-1}dr^2 + r^2d\omega^2,\]
where $r=r_s$ represents the event horizon and $d\omega^2$ is the line
element on the $(n-1)$-dimensional unit sphere.  Here $d=n-3$, which
allows us to quickly compare to the $n=3$ case.  We define
$r_{ps}=\Bigl(\frac{d+3}{2}\Bigr)^{\frac{1}{d+1}}r_s$ as $r=r_{ps}$ is
the location of the photon sphere.  $K=\partial_t$ is the Killing
vector field that is timelike in the domain of outer communication.
It extends into the interior of the black hole, becoming null on the
event horizon and spacelike in the interior.

As in the more typical $(1+3)$-dimensional case, 
the singularity at $r=r_s$ is a coordinate singularity.  Setting 
\[r^* = \int_{r_{ps}}^r
\Bigl(1-\Bigl(\frac{r_s}{s}\Bigr)^{d+1}\Bigr)\, ds\]
and $v=t+r^*$, the metric becomes
\[ds^2 = -\weight dv^2 + 2\,dv\,dr + r^2 d\omega^2.\]
While these coordinates are nonsingular at the event horizon, surfaces
of constant $v$ are null and not convenient for posing Cauchy problems
for the wave equation.  To this end, we introduce, as in \cite{MMTT},
\[\tv = v-\mu(r), \quad \mu\in C^\infty\]
subject to
\begin{itemize}
\item $\mu(r)\ge r^*$ for $r>r_s$ and $\mu(r)=r^*$ for
  $r>(r_s+r_{ps})/2$,
\item  $\mu'(r)>0$ and $2-\weight \mu'(r) >0$.
\end{itemize}
The metric, in the $(\tv, r,\omega)$ coordinates, then becomes
\begin{multline}\label{Smetric}ds^2 = -\weight\, d\tv^2 + 2\Bigl(1-\weight \mu'(r)\Bigr)\,d\tv\,dr
\\+ \Bigl(2\mu'(r)-\weight (\mu'(r))^2\Bigr)\,dr^2 + r^2\,d\omega^2.
\end{multline}
The first condition on $\mu$ shows that these coordinates coincide
with the Schwarzschild coordinates away from the event horizon.  The
second condition guarantees that $\tilde{v}=$constant slices are space-like.

For a choice $0<r_e<r_s$, we seek to solve the wave equation 
\begin{equation} \label{Swave}\Box_S u = f,\end{equation}
where $\Box_S = \nabla^\alpha \partial_\alpha$ when the background
metric is given by \eqref{Smetric},
in 
\[\M_R = \{\tv\ge 0, r\ge r_e\}\]
with initial data on 
\[\Sigma^-_R = \M_R\cap \{\tv=0\},\]
which is space-like.  We use 
\[\Sigma^+_R = \M_R\cap\{r=r_e\}\]
to denote the lateral boundary.

We shall work with a nondegenerate energy, rather than the conserved
energy associated to the Killing vector field $\partial_t$, as was
used in \cite{LM}.  This will help us, as in \cite{MMTT}, to cleanly
utilize the red-shift effect.  We define the energy on an arbitrary
$\tv=\tv_0$ slice to be
\[E[u](\tv_0) = \int_{\M_R\cap \{\tv=\tv_0\}} \Bigl( (\partial_r u)^2
+ (\partial_{\tv} u)^2 + |\ang u|^2\Bigr)\,r^{d+2}\,dr\,d\omega.\]
The initial energy $E[u](\Sigma^-_R)$ and the outgoing energy $E[u](\Sigma^+_R)$ are
defined as
\[E[u](\Sigma^-_R)=E[u](0),\quad E[u](\Sigma_R^+)=\int_{\Sigma^+_R} \Bigl((\partial_r u)^2
+ (\partial_{\tv} u)^2 + |\ang u|^2\Bigr)\,r_e^{d+2}\,d\tv\,d\omega.\]

We modify the localized energy spaces to reflect the quadratic loss at
the photon sphere.  To do so, we set
\[\|u\|_{LE_S} = \Bigl\|\Bigl(\frac{r-r_{ps}}{r}\Bigr) u\Bigr\|_{LE_M},\]
and
\[\|u\|_{LE^1_S} = \|\partial_r u\|_{LE_M} + \|\partial_{\tv}
u\|_{LE_S} + \|\ang u\|_{LE_S} + \|r^{-3/2} u\|_{L^2(\M_R)}.\]
Note that the loss due to trapping only occurs on the $\partial_{\tv}$
and $\ang$ components.  We analogously set
\[\|f\|_{LE^*_S} = \Bigl\|\Bigl(\frac{r-r_{ps}}{r}\Bigr)^{-1} f\Bigr\|_{LE^*_M}.\]

Then, we have the following theorem, which is largely from \cite{LM}
and \cite{Schlue}.

\begin{theorem}\label{thmSchwarzschild}
For $u$ solving \eqref{Swave}, we have
\[E[u](\Sigma^+_R) + \sup_{\tv} E[u](\tv) + \|u\|^2_{LE^1_S} \lesssim
E[u](\Sigma^-_R) + \|f\|^2_{LE^*_S}.\]
\end{theorem}

For any Lorentzian metric $g$, the energy-momentum tensor for the wave
equation is given by
\[Q_{\alpha\beta}[u]=\partial_\alpha u \partial_\beta u -
\frac{1}{2}g_{\alpha\beta}\partial^\gamma u\partial_\gamma u,\]
satisfying $\nabla^\alpha Q_{\alpha\beta}[u] = \partial_\beta u
\Box_g u$.  For a $C^1$ vector field $X$, a
scalar field $q$, and a 1-form $m$, we set
\[P_\alpha[u,X,q,m] = Q_{\alpha\beta}[u]X^\beta + q u\partial_\alpha u
-\frac{1}{2}(\partial_\alpha q) u^2 + \frac{1}{2}m_\alpha u^2.\]
It, thus, follows that
\begin{equation}\label{divP}\nabla^\alpha P_{\alpha}[u,X,q,m] = \Box_g u \Bigl(Xu+qu\Bigr) +
Q[u,X,q,m],
\end{equation}
where
\[Q[u,X,q,m] = Q_{\alpha\beta}[u]\pi^{\alpha\beta} + q\partial^\alpha
u \partial_\alpha u + m_\alpha u \partial^\alpha u  +
\frac{1}{2}(\nabla^\alpha m_\alpha - \nabla^\alpha \partial_\alpha q)u^2.\]
Here
\[\pi_{\alpha\beta}=\frac{1}{2}(\nabla_\alpha X_\beta +\nabla_\beta
X_\alpha)\]
denotes the deformation tensor of $X$.

The theorem will follow from a proper choice of $X$, $q$, and $m$.
Indeed, we will construct such satisfying the following lemma.

\begin{lemma}\label{mmttLemma}
There exist $X$, $q$, and $m$ that are smooth, spherically symmetric,
and $K$-invariant in $r\ge r_s$.  Moreover, $X$ is bounded in
$(\tv,r,\omega)$ coordinates and $X(dr)(r_s)<0$; $|q(r)|\lesssim
r^{-1}$ and $|q'(r)|\lesssim r^{-2}$; $m$ has compact support in $r$
and $\la m,dr\ra(r_s)>0$; and 
\[Q[u,X,q,m]\gtrsim r^{-(d+3)} (\partial_r u)^2 +
\Bigl(\frac{r-r_{ps}}{r}\Bigr)^2(r^{-(d+3)} (\partial_\tv u)^2 + r^{-1}
|\ang u|^2) + r^{-3} u^2.\]
\end{lemma}

\begin{proof}
 We begin by working in the $(t,r,\omega)$ coordinates and set
\[X=X_1+\delta X_2,\quad \delta\ll 1\]
with
\[X_1 = f(r)\weight \partial_r,\quad X_2 = - b(r)\Bigl(\partial_r
- \weight^{-1}\partial_t\Bigr).\]
For smooth choices of $f$ and $b$, these are smooth on
$\M_R$ in the $(v,r,\omega)$ coordinates.  Indeed,
\[X_1 = f(r) \Bigl(\weight \partial_r + \partial_v\Bigr),\quad X_2 = -b(r)\partial_r.\]

The choice of $X_1$ is inspired by the preceding work \cite{LM},
though to ease calculations in the sequel, we shall mollify the
original multiplier near the photon sphere.  The vector field $X_2$ is
inspired by \cite{DR}.  See \cite{Schlue} for the corresponding
argument in generic dimension.  This component allows us to take
advantage of the red shift effect.

To begin, for $g(r)=\frac{r^{d+2}-r_{ps}^{d+2}}{r^{d+2}}$ and
$h(r)=\ln\Bigl(\frac{r^{d+1}-r_s^{d+1}}{\frac{d+1}{2}r_s^{d+1}}\Bigr)$,
we set
\[f_1(r)=g(r)+\frac{d+2}{d+3}\frac{r_{ps}r_s^{d+1}}{r^{d+2}}a(h(r)),\]
where
\[a(x)=
\begin{cases}
%  -\frac{1}{\varepsilon}\frac{\varepsilon x+1}{\delta(\varepsilon
%    x+1)-1} - \frac{1}{\varepsilon},&x\le -\frac{1}{\varepsilon},\\
x,& x\le 0,\\%-\frac{1}{\varepsilon} \le x\le 0,\\
x-\frac{2}{3\alpha^2}x^3 + \frac{1}{5\alpha^4}x^5,&0\le x\le \alpha,\\
\frac{8\alpha}{15},&x\ge \alpha.
\end{cases}\]
Here $\alpha =
5-$.  The function $a$ serves to
smooth out the logarithm near infinity to insure that $f_1$ is
everywhere increasing.  For later purposes, we note that $a'''$ is
everywhere non-positive.
With the further choice of
\[q_1[f_1] = \frac{1}{2}\Bigl[\weight r^{-(d+2)}\partial_r (r^{d+2}
f_1(r))\Bigr],\]
it was shown in \cite{LM} that 
\begin{multline*}Q[u, X_1[f_1], q_1,0] \gtrsim \weight
r^{-(d+3)}(\partial_r u)^2 +
\frac{1}{r}\Bigl(\frac{r-r_{ps}}{r}\Bigr)^2 |\ang u|^2 \\ + r^{-3} \weight u^2
 \end{multline*}
where $X_1[f_1] = f_1 \weight \partial_r$.

We would like to replace $f_1$ by a similar multiplier $F$ that is smooth at the
photon sphere. In order for $Q[u, X_1[f_1], q_1,0]$ to still satisfy the inequality above, it is sufficient to pick $F$ to be increasing, bounded, and $l(F)>0$\footnote{$l(f_1)$ is the coefficient of $u^2$ in $Q[u, X_1[f_1], q_1,0]$} where 
\[l(F) = -\frac{1}{4}r^{-(d+2)}\partial_r\Bigl[\weight
r^{d+2} \partial_r \Bigl\{ \weight r^{-(d+2)}\partial_r
(F(r)r^{d+2})\Bigr\}\Bigr].\]

 Let
\begin{align*}F(r) &= f_1(r)+\frac{d+2}{d+3} \frac{r_{ps}r_s^{d+1}}{r^{d+2}} \chi(r)
\Bigl( (\psi_N * a)(h(r))-Q_2(r)-a(h(r))\Bigr)\\
&=g(r) + \frac{d+2}{d+3}\frac{r_{ps}r_s^{d+1}}{r^{d+2}}\Bigl[
(1-\chi)(r)a(h(r)) + \chi(r)\Bigl((\psi_N*a)(h(r)) - Q_2(r)\Bigr)\Bigr].
\end{align*}
Here $\psi_N$ is a standard mollifier, which is an approximation of
the identity for large $N$.  See, e.g., \cite{Evans}.  $Q_2$ is a second order polynomial
(with very small coefficients) that is selected so that all
derivatives up to order $2$ of $(\psi_n * a)(h(r))-Q_2(r)$ coincide
with those of $a(h(r))$ at $r=r_{ps}$.  The function $\chi(r)$ is a
smooth cutoff that is supported in a small neighborhood of $r_{ps}$
and is the identity on a smaller neighborhood of $r_{ps}$.
Since $a\in C^2$ on the support of $\chi$, given any $\varepsilon>0$,
we can choose $N$ sufficiently large so that
\[|\partial^\alpha(F(r)-f_1(r))|< \varepsilon, \quad 0\le \alpha\le 2.\]

We proceed to show that $F$ preserves the desired properties of $f_1$ on
the support of $\chi$ if $N$ is chosen sufficiently large.  By
construction, we have that $F(r_{ps})=f_1(r_{ps})=0$.  By choosing $N$
sufficiently large so that
\[|\partial(F(r)-f_1(r))| < \frac{1}{2}\inf_{r\in \supp\chi} f_1'(r),\]
it follows easily that $F'(r)>0$.

It finally remains to show that $l(F)>0$. We write
\begin{multline}\label{lF}l(F)= l(f_1) + \Bigl[l(F-f_1) +
\frac{1}{4}\frac{d+2}{d+3}\frac{r_{ps}r_s^{d+1}}{r^{d+2}} \chi(r)
\weight^2 (\psi_N * a)'''(h(r)) (h'(r))^3 \Bigr] \\-
\frac{1}{4}\frac{d+2}{d+3}\frac{r_{ps}r_s^{d+1}}{r^{d+2}} \chi(r)
\weight^2 (\psi_N * a''')(h(r)) (h'(r))^3.
\end{multline}
As
\[\Bigl[l(F-f_1) +
\frac{1}{4}\frac{d+2}{d+3}\frac{r_{ps}r_s^{d+1}}{r^{d+2}} \chi(r)
\weight^2 (\psi_N * a)'''(h(r)) (h'(r))^3 \Bigr]\]
 only contains derivatives up to
order $2$ of the mollified $a$, it can be made uniformly small, and an
argument similar to that employed to show that $F'(r)>0$ can be used
to show $l(F)>0$ provided that 
\[-
\frac{1}{4}\frac{d+2}{d+3}\frac{r_{ps}r_s^{d+1}}{r^{d+2}} \chi(r)
\weight^2 (\psi_N * a''')(h(r)) (h'(r))^3>0,\]
which follows easily from the non-positivity of $a'''$ on the relevant region.

This gives that
\begin{multline*}Q[u, X_1[F], q_1[F],0] \gtrsim \weight
r^{-(d+3)}(\partial_r u)^2 +
\frac{1}{r}\Bigl(\frac{r-r_{ps}}{r}\Bigr)^2 |\ang u|^2 \\ + r^{-3} \weight u^2
 \end{multline*}
for a smooth choice of $F$.

This choice of $F$, however, does not give a bounded multiplier at the
event horizon.  To rectify this, we set $\rho$ to be a smooth, increasing function
satisfying
\[\rho(R)=
\begin{cases}
  R,\quad R\ge -1,\\
  -2,\quad R\le -3,
\end{cases}\]
and let $\rho_\varepsilon(R) = \varepsilon^{-1} \rho(\varepsilon
R)$.  Finally, we set
\[f(r) = \frac{1}{r^{d+2}} \rho_\varepsilon(r^{d+2}F(r))\]
in the definition of $X_1$.  Recalculating, with the abbreviated
notation $X_1[f]=X_1$ and $q_1[f]=q_1$, we have
\[  Q[u, X_1, q_1, 0] \ge C\weight^2 r^{-(d+3)} (\partial_r u)^2 +
  C\frac{1}{r} \Bigl(\frac{r-r_{ps}}{r}\Bigr)^2 |\ang u|^2 + l(f).
\]
As before, $l(f)\gtrsim r^{-3} \weight$ except when $\varepsilon r^{d+2} F(r) <-
1$.  We record that
\[l(f) = \rho'(\varepsilon r^{d+2}F(r)) l(F) - O(\varepsilon)
\rho''(\varepsilon r^{d+2}F(r)) - O(\varepsilon^2 \weight^{-1})
\rho'''(\varepsilon r^{d+2}F(r)).\]

For the multiplier $X_2$, we calculate
\begin{multline*}Q\Bigl[u,X_2,-\frac{d+2}{2}\frac{b(r)}{r}
,0\Bigr]=-\frac{1}{2}b'(r)\weight \Bigl(\partial_r u - \weight^{-1} \partial_t
u\Bigr)^2 
\\+\frac{(d+1)r_s^{d+1}}{2} \frac{b(r)}{r^{d+2}}\Bigl(\partial_r u - \weight^{-1} \partial_t
u\Bigr)^2 
\\+\Bigl[\frac{1}{2}b'(r)-\frac{b(r)}{r}\Bigr]\ang u|^2
+ \frac{d+2}{4}\nabla^\alpha \partial_\alpha\Bigl(\frac{b(r)}{r}
\Bigr) u^2. 
%+ \Bigl(X_2 u + \frac{1}{r}\weight
%b(r) u\Bigr)\Box_g u.
\end{multline*}
Inspired by this, we choose $b(r)$ to be smooth, positive and decreasing on $[r_s, (r_s+3r_{ps})/4)$, and
supported in $[r_s, (r_s+3r_{ps})/4]$.  For later purposes, we set
\[m_t = \frac{(d+1)r_s^{d+1}}{r^{d+2}}b(r)\gamma,\quad m_r = 
\weight^{-1}\frac{(d+1)r_s^{d+1}}{r^{d+2}}b(r)\gamma,\quad m_\omega = 0.\]
Then,
\begin{multline*}
  Q\Bigl[u,X_2, -\frac{d+2}{2}\frac{b(r)}{r}, m\Bigr] \ge 
%-\frac{1}{2}b'(r)\weight \Bigl(\partial_r u - \weight^{-1} \partial_t
%u\Bigr)^2 
%\\+
\frac{(d+1)r_s^{d+1}}{2} \frac{b(r)}{r^{d+2}}\Bigl(\partial_r u - \weight^{-1} \partial_t
u + \gamma u\Bigr)^2 
\\+\Bigl[\frac{1}{2}b'(r)-\frac{b(r)}{r}\Bigr]\ang u|^2 -
\frac{(d+1)r_s^{d+1}}{2}\frac{b(r)}{r^{d+2}} \gamma^2 u^2
\\+ \frac{d+2}{4}\nabla^\alpha \partial_\alpha\Bigl(\frac{b(r)}{r}
\Bigr) u^2  + (d+1)\frac{r_s^{d+1}}{r^{d+2}}\partial_r (b(r)\gamma)
u^2
\end{multline*}

Combining, we obtain
\begin{multline}\label{Q1}
  Q\bigl[u, X_1+\delta X_2,
  q_1-\delta\frac{d+2}{2}\frac{b(r)}{r},\delta m]
  \ge C \weight^2 r^{-(d+3)}(\partial_r u)^2 \\+ C \frac{1}{r}
  \Bigl(\frac{r-r_{ps}}{r}\Bigr)^2 |\ang u|^2
+ C \delta b(r) \Bigl(\partial_r u - \weight^{-1} \partial_t
u + \gamma u\Bigr)^2 
+ n(r) u^2
\end{multline}
provided $\delta>0$ is sufficiently small.  Here
\begin{multline*}
  n(r) = -\delta\frac{(d+1)r_s^{d+1}}{2}\frac{b(r)}{r^{d+2}} \gamma^2
+ \delta \frac{d+2}{4}\nabla^\alpha \partial_\alpha\Bigl(\frac{b(r)}{r}
\Bigr)  +\delta (d+1)\frac{r_s^{d+1}}{r^{d+2}}\partial_r (b(r)\gamma) + l(f).
\end{multline*}
Before proceeding to bound $n$ from below, we address the fact that 
the right side of \eqref{Q1}
fails to control $\partial_t$ away from the event
horizon.  This is easily remedied by setting
\[q_2(r) = \chi_{r>(r_s+r_{ps})/2} r^{-(d+3)} \Bigl(\frac{r-r_{ps}}{r}\Bigr)^2,\]
and recomputing
\[Q\Bigl[u, X_1+\delta X_2, q_1-\delta \frac{d+2}{2}\frac{b(r)}{r} +
\delta_1 q_2,\delta m\Bigr].\]
Provided that $\delta_1\ll \delta$ and bootstrapping the many
negligible terms, we obtain
\begin{multline*}
  Q\Bigl[u, X_1+\delta X_2, q_1-\frac{d+2}{2}\frac{b(r)}{r}+
  \delta_1 q_2,m\Bigr]
  \ge C\weight^2 r^{-(d+3)}(\partial_r u)^2 \\+ C\frac{1}{r}
  \Bigl(\frac{r-r_{ps}}{r}\Bigr)^2 |\ang u|^2
+ C\delta b(r) \Bigl(\partial_r u - \weight^{-1} \partial_t
u + \gamma u\Bigr)^2 \\+ \delta_1 q_2 \weight^{-1} (\partial_t u)^2
+ n(r) u^2.
\end{multline*}

As we may easily observe that
\[X(dr)(r_s) = -\delta b(r_s) <0,\quad \la m,dr\ra(r_s) =
\frac{d+1}{r_s}b(r_s)\gamma(r_s) >0,\]
provided that $\gamma(r_s)>0$, it remains to choose $\gamma$ so that
$n$ is positive where $\varepsilon r^{d+2}f(r)<-1$.  We suppose that
$\supp\gamma\subset \{r<r_{ps}\}$, $0\le \gamma\le 1$, and $\gamma'>-1$.
For $r>r_{ps}$, the lower bound follows from that for $l(f)$ by the
support properties of $b(r)$ and $\gamma$.  For $r<r_{ps}$, we write 
\[n = l(f) + \delta \frac{(d+1)r_s^{d+1}}{r^{d+2}} b(r) \gamma'(r) -
O(\delta).\]
The lower bound on $\gamma'$, then, implies that $n\ge
l(f)-O(\delta)$, which is positive where $\varepsilon r^{d+2}f(r)\ge
-1$ for sufficiently small $\delta$.

Finally, in the region $\varepsilon r^{d+2}f(r)<-1$,
\begin{multline*}n \ge \delta \frac{(d+1)r_s^{d+1}}{r^{d+2}} b(r) \gamma'(r) -
O(\delta)- O(\varepsilon)
\rho''(\varepsilon r^{d+2}F(r)) \\- O(\varepsilon^2 \weight^{-1})
\rho'''(\varepsilon r^{d+2}F(r)).\end{multline*}
While $\gamma'$ may be taken to be positive, the requirement that
$0\le \gamma \le 1$ places restrictions on how large it may be.  On
average, $\gamma'$ can be at most $O(e^{c/\varepsilon})$ on this
interval.  But since the interval of integration $\varepsilon
r^{d+2}f(r)<1$ is of length $e^{-c/\varepsilon}$,
\[\int_{\varepsilon r^{d+2}f(r)<-1} \delta + \varepsilon
|\rho''(\varepsilon r^{d+2}f(r))| + \varepsilon^2 \weight^{-1}
|\rho'''(\varepsilon r^{d+2}f(r))|\,dr \lesssim \varepsilon +
e^{-c/\varepsilon} \ll \delta,\]
provided that $\varepsilon$ is sufficiently small.  And this completes
the proof of the lemma.
\end{proof}

\begin{proof}[Proof of Theorem \ref{thmSchwarzschild}]  
With the previous lemma in hand, we only tersely describe the
remainder of the proof as the details follow from obvious
modifications of the argument in \cite{MMTT}.
We allow $X$,
  $q$, and $m$ to be as in the lemma.
  From the divergence relation \eqref{divP} and the fact that $K$ is a
  Killing vector field, we have
\[\nabla^\alpha P_\alpha[u, X+CK, q, m] = (\Box_g u)((X+CK)u+qu) +
Q[u,X,q,m]\]
for a large constant $C$.  Integrating over $0<\tv<\tv_0$, and
$r>r_e$, we obtain
\begin{multline*}
  \int \Bigl(\Box_g u  ((X+CK)u+qu) + Q[u,X,q,m])\,dVol = \int \la
  d\tv, P[u, X+CK,q,m]\ra r^{d+2}\,dr\,d\omega\Bigl|_{\tv = 0}^{\tv =
    \tv_0}\\ - \int_{r=r_e} \la dr, P[u,X+CK,q,m]\ra\, r_e^{d+2}\,d\tv d\omega.
\end{multline*}

For $C$ large enough and $r_e$ near enough to $r_s$, we claim
\begin{equation}
  \label{timeBdy}
  E[u](\tv_1) \approx -\int_{\tv=\tv_1} \la
  d\tv,P[u,X+CK,q,m]\ra\,r^{d+2}\,dr\,d\omega,\quad \tv_1\ge 0
\end{equation}
and
\begin{equation}
  \label{lateralBdy}
  \la dr,P[u,X+CK,q,m]\ra \gtrsim (\partial_r u)^2 + (\partial_\tv
  u)^2 + |\ang u|^2 + u^2,\quad r=r_e.
\end{equation}
Upon proving these, a localized energy estimate, though with weight
$r^{-(d+3)/2}$ rather than $r^{-1/2-}$ at infinity on the $\partial_r$
and $\partial_{\tv}$ terms, follows immediately from the lemma and the
Schwarz inequality.

To prove \eqref{timeBdy}, we note that, for $r>r_s$,
\begin{align*}
  \la d\tv, P[u,X+CK]\ra &= (X(d\tv)+C)\la d\tv, P[u,\partial_{\tv}]\ra
  + X(dr)\la d\tv,P[u,\partial_r]\ra\\
&\approx C\Bigl[(\partial_{\tv} u)^2 + \weight (\partial_r u)^2 +
|\ang u|^2\Bigr] + (\partial_r u)^2.
\end{align*}
Indeed, due to the boundedness of $X$, $X(d\tv)+C\approx C$ for $C$
large enough.  The condition $X(dr)(r_s)<0$ is used to obtain the
nondegenerate $\partial_r$ contribution.  And by continuity, the same
expression extends to $r>r_e$ for $r_e$ sufficiently close to $r_s$.
To complete the proof of \eqref{timeBdy}, we note that the decay of
$q$, the compact support of $m$, and the Schwarz inequality reduce
controlling the lower order terms to proving a straightforward Hardy-type inequality
such as
\[\int_{r_e}^\infty r^{-2} u^2\, r^{d+2}\,dr \lesssim
C^{-1/2}\int_{r_e}^\infty \Bigl[C\weight +1\Bigr] (\partial_r u)^2 \,r^{d+2}\,dr.\]

For \eqref{lateralBdy}, we similarly compute
\[\la dr, P[u, X+CK, 0, m]\ra \gtrsim C (\partial_\tv u)^2 + |\ang
u|^2 - \Bigl(1-\frac{r_s^{d+1}}{r_e^{d+1}}\Bigr) (\partial_r u)^2 +
u^2,\]
when $0<r_s-r_e \ll 1$.  Here we have used the lateral boundary
condition on $m$ as provided by the lemma.  As
\[q u \la dr, du\ra - \frac{1}{2}u^2\la dr,dq\ra \ll
C(\partial_\tv u)^2 + u^2 + \Bigl(1-\frac{r_s^{d+1}}{r_e^{d+1}}\Bigr)
(\partial_r u)^2,\]
provided that $C$ is large enough and $r_s-r_e$ is sufficiently small,
these terms can be bootstrapped into the above. We can now translate these bounds to any $0<r_e<r_s$ from local thoery. This completes the
proof of the theorem with the weaker weights at infinity. 

In order to get the sharp weights at infinity, we can rely on the
methods of \cite{MS} for small perturbations of Minkowski space.  See
the sharper statement of the results that follow from that proof in
\cite[Lemma 3.3]{LMSTW}.  Indeed, for some constant $R\gg r_{ps}$ and
a smooth $\beta$ so that $\beta(r)\equiv 0$ for $r<R$ and
$\beta(r)\equiv 1$ for $r>2R$, then an application of the results of
\cite{MS} complete the proof provided that
\[\|[\Box_g, \beta] u\|_{LE^*_M},\]
can be bounded.  But as $[\Box_g,\beta]$ is supported in a compact
region, it follows immediately that the above is controlled using the
previously established estimate with weaker weights at infinity.
\end{proof}

%%%%%%%%%%%%%%%%%%%%%%%%%%%%%%%%%%%%%%%%%%%%%%%%%%%%%%%%%%%%%%%%%%%%%%%%%%%%%%%%
%%%%%%%%%%%%%%%%%%%%%%%%%%%%%%%%%%%%%%%%%%%%%%%%%%%%%%%%%%%%%%%%%%%%%%%%%%%%%%%%
%%%%%%%%%%%%%%%%%%%%%%%%%%%%%%%%%%%%%%%%%%%%%%%%%%%%%%%%%%%%%%%%%%%%%%%%%%%%%%%%
\newsection{The Metric}\label{theMetric}
Here we introduce the class of
(1+4)-dimensional Myers-Perry black holes.  These space-times are
higher dimensional analogs of the well-known Kerr solutions of
Einstein's equations.  With the additional dimension, an extra angular
momentum parameter is introduced.  

The line element for the $(1+4)$-dimensional Myers-Perry solution is given by
\begin{multline}\label{MPmetric} ds^2 = -dt^2 +\frac{r_s^2}{\rho^2}\Bigl[dt+a\sin^2\theta d\phi + b
\cos^2\theta d\psi\Bigr]^2  +\frac{\rho^2}{4\Delta} dx^2 \\+ \rho^2
d\theta^2 + (x+a^2)\sin^2\theta d\phi^2 + (x+b^2)\cos^2\theta d\psi^2.
\end{multline}
Here $r_s$ is a parameter, depending on the mass of the black hole,
which corresponds to the Schwarzschild radius.  Moreover,
\begin{align*}\rho^2 &= x+a^2\cos^2\theta + b^2 \sin^2\theta,\\
\Delta &= (x+a^2)(x+b^2) - r_s^2 x.
\end{align*}
The parameters $a,b$ are angular momentum parameters.  We are using
the notational conventions that we learned from \cite{FS}.  In
particular, $x=r^2$, where $r$ is the commonly used ``radial'' variable.

The event horizons are given by the solutions to $\Delta=0$.  In
particular, they are $x=x_\pm$,
\[x_\pm = \frac{1}{2}\Bigl(r_s^2 - a^2 - b^2 \pm \sqrt{(r_s^2 -
  a^2-b^2)^2-4a^2b^2}\Bigr).\]
We shall primarily be concerned with the domain of outer communication $x>x_+$.

For future reference, we note that
\[\sqrt{-g} = \frac{1}{2}\sin\theta \cos \theta \rho^2\]
and that the 
inverse metric is given by
\[g^{tt} = \frac{1}{\rho^2}\Bigl[(a^2-b^2)\sin^2\theta -
\frac{(x+a^2)[\Delta + r_s^2(x+b^2)]}{\Delta}\Bigr],\]
\[g^{t\phi} = \frac{ar_s^2 (x+b^2)}{\rho^2 \Delta},\quad g^{t\psi} =
\frac{br_s^2(x+a^2)}{\rho^2\Delta}\]
\[g^{\phi\phi} = \frac{1}{\rho^2}\Bigl[\frac{1}{\sin^2\theta} -
\frac{(a^2-b^2)(x+b^2)+b^2r_s^2}{\Delta}\Bigr]\]
\[g^{\psi\psi} = \frac{1}{\rho^2} \Bigl[\frac{1}{\cos^2\theta}
+\frac{(a^2-b^2)(x+a^2)-a^2r_s^2}{\Delta}\Bigr]\]
\[g^{\phi\psi} = -\frac{ab r_s^2}{\rho^2\Delta},\quad
g^{xx}=4\frac{\Delta}{\rho^2},\quad g^{\theta\theta} = \frac{1}{\rho^2}.\]

As expected, the Myers-Perry black holes serve as higher dimensional
analogs to the axisymmetric Kerr black holes.  The structure of the
singularities, horizons, and ergoregions display similarities to those
of the more familiar Kerr space-times.  Further similarities include
superradiance, the construction of maximal analytic extensions, and 
the existence of hidden symmetries.  See \cite{MP}, \cite{Myers},
\cite{FK}, and the references therein.  Some properties do differ,
such as the existence of ultra-spinning black holes \cite{MP},
\cite{Myers}. 

 In what follows, we focus our comparisons on the
behavior of null geodesics.  In particular, we show that each trapped geodesics has a
constant $r$.  Like the Carter tensor \cite{Carter2} for Kerr, the Myers-Perry
space-times contain hidden symmetries \cite{FK} in all dimensions, but
as is mentioned therein this is insufficient to guarantee the
separability of the Hamilton-Jacobi equations in generic dimensions.
The Hamilton-Jacobi equations can be separated in the special case
that there are, at most, two distinct rotations parameters.  See \cite{Palmer}
and \cite{VSP}.  Specifically for the $(1+4)$-dimensional solutions, the
hidden symmetries and separability appear in \cite{FS}, and it is this
exposition that closely informs what immediately follows.

  Let $E=-\partial_t$, $\Phi = \partial_\phi$, $\Psi = \partial_\psi$,
and $K$ denote the constants of motion.  Then we have, again with
notations from \cite{FS},
\begin{align}
  \rho^4 \dot{\theta}^2 &= \Theta,\label{dotTheta}\\
\rho^4 \dot{x}^2 &= 4\mathcal{X},\label{dotx}\\
\rho^2 \dot{t} &= E\rho^2 + \frac{r_s^2 (x+a^2)(x+b^2)}{\Delta}\mathcal{E},\\
\rho^2 \dot{\phi} &= \frac{\Phi}{\sin^2\theta} - \frac{a r_s^2
  (x+b^2)}{\Delta}\mathcal{E} - \frac{(a^2-b^2)\Phi}{x+a^2},\\
\rho^2 \dot{\psi} &= \frac{\Psi}{\cos^2\theta} - \frac{br_s^2 (x+a^2)}{\Delta}\mathcal{E}+\frac{(a^2-b^2)\Psi}{x+b^2}.
\end{align}
Here we have abstracted
\begin{align}
\mathcal{E}&= E+\frac{a\Phi}{x+a^2} + \frac{b\Psi}{x+b^2} \\
\Theta &= E^2(a^2\cos^2\theta +
b^2\sin^2\theta)-\frac{1}{\sin^2\theta}\Phi^2 -
\frac{1}{\cos^2\theta}\Psi^2 + K\label{Theta}\\
\mathcal{X}&=\Delta \Bigl[E^2 x
+(a^2-b^2)\Bigl(\frac{\Phi^2}{x+a^2}-\frac{\Psi^2}{x+b^2}\Bigr)-K\Bigr]
+ r_s^2 (x+a^2)(x+b^2)\mathcal{E}^2.\label{X}
\end{align}

% {\color{blue} The sign on the $\theta$ term seems to mismatch what
%   Mihai has in his thesis... compare the $E^2$ and $K$ terms.
%   Different convention?  $K$ convention seems to differ by a constant, which is fine.}

By writing $(a^2-b^2) =
(x+a^2)-(x+b^2)$, we see that
\begin{multline*}-r_s^2 x (a^2-b^2)\Bigl(\frac{\Phi^2}{x+a^2} -
\frac{\Psi^2}{x+b^2}\Bigr) + r_s^2 (x+a^2)(x+b^2)\Bigl(\frac{a^2
  \Phi^2}{(x+a^2)^2} + \frac{b^2\Psi^2}{(x+b^2)^2}\Bigr)
\\= r_s^2(b^2\Phi^2 + a^2 \Psi^2).\end{multline*}
Making this substitution in the definition of $\mathcal{X}$, we obtain
\begin{multline*}\mathcal{X} = \Delta (E^2 x -K) + (a^2-b^2)\Bigl(\Phi^2 (x+b^2) - \Psi^2
(x+a^2)\Bigr) \\+ r_s^2 \Bigl(E^2(x+a^2)(x+b^2) + 2aE\Phi(x+b^2)+2bE\Psi(x+a^2)+(b\Phi+a\Psi)^2\Bigr).
%2ab\Phi\Psi\Bigr)\\+ r_s^2 (b^2\Phi^2 + a^2\Psi^2).
\end{multline*}
We have reduced our analysis to understanding this cubic polynomial as
the zeros of $\mathcal{X}$ correspond to turning points of \eqref{dotx}.

At least one of the parameters $E, K, \Phi, \Psi$ should be nonzero.
As \eqref{dotTheta} disallows $E=K=0$, the above representation for
$\mathcal{X}$ is always nondegenerate.  

We first examine the case that $E=0$.  Rewriting \eqref{Theta}
\[\Theta = \Bigl[K+(\Phi-Ea)^2 -\Phi^2 + (\Psi-Eb)^2-\Psi^2\Bigr] -
\frac{(\Phi - Ea\sin^2\theta)^2}{\sin^2\theta} -
\frac{(\Psi-Eb\cos^2\theta)^2}{\cos^2\theta},\]
it follows from \eqref{dotTheta} that
\[K+(\Phi-Ea)^2 -\Phi^2 + (\Psi-Eb)^2-\Psi^2 \ge 0.\]
Thus, in particular, when $E=0$, we must have $K>0$.
In this case,
$\mathcal{X}$ is a quadratic equation with leading coefficient $-K$,
and we note that $\mathcal{X}(x_+)\ge 0$ since $\Delta(x_+)=0$.  Thus, since
$K>0$, there is a single zero with multiplicity 1 outside of the
event horizon.  $\mathcal{X}$ must change sign at this point from
positive to negative, which
creates a single right turning point for \eqref{dotx} and no trapped
geodesic can exist. 
%  If $K<0$, then either $\mathcal{X}$ has no roots
% greater that $x_+$, has two distinct roots in $x>x_+$, or has a single
% double root in this region.  If there are no roots, then
% $\mathcal{X}>0$ in this region and all null geodesics escape to
% infinity.  If there are distinct roots $x_1$ and $x_2$ so that
% $x_+<x_1<x_2$, then $\mathcal{X}$ changes sign from positive to
% negative and negative to positive respectively at $x_1$ and $x_2$.
% Thus $x_1$ is a right turning point and $x_2$ is a left turning point,
% which does not support a trapped null geodesic.  A double root $x_0$ of
% $\mathcal{X}$ corresponds to a steady state and implies that all other geodesics
% converge to $x_0$ in one direction and either $x=0$ or $\infty$ in the
% other direction, which are thus not trapped.

% {\color{blue} There is a proof in O'Neill's book that $K>0$.  Seems
%   likely that the $K<0$ is irrelevant. }

We next examine the case $E\neq 0$.  Here $\mathcal{X}$ is a cubic polynomial
with a positive leading coefficient.  As it is obvious from the
original formulation \eqref{X} of $\mathcal{X}$ that $\mathcal{X}(x_+)>0$, it
follows that there is always one root of $\mathcal{X}$ to the left of
$x_+$.  Three cases then must be examined: (1) $\mathcal{X}$ has no
roots greater than $x_+$, (2) it has two distinct roots in $x>x_+$, or
(3) it has a since double root in this region.  If there are no roots,
then $\mathcal{X}>0$ in this region and all null geodesics escape to
infinity.  If there are two distinct roots $x_1, x_2$ so that
$x_+<x_1<x_2$, then $\mathcal{X}$ changes sign from positive to
negative and negative to positive respectively at $x_1$ and $x_2$.
Thus $x_1$ is a right turning point and $x_2$ is a left turning point,
which does not support a trapped geodesic.  A double root $x_0$ of
$\mathcal{X}$ corresponds to a steady state and implies that all other
geodesics converge to $x_0$ in one direction and either $x=0$ or
$\infty$ in the other direction, which are thus not trapped.

% The remaining cases for the possibilities of roots, then,
% directly mimics the $K<0$ possibilities above.  In particular, the
% only trapped geodesics occur at roots of $\mathcal{X}$.

The above analysis shows that along any trapped null geodesic we have
constant $x$, which corresponds to a double root of $\mathcal{X}$.

%%%%%%%%%%%%%%%%%%%%%%%%%%%%%%%%%%%%%%%%%%%%%%%%%%%%%%%%%%%%%%%%%%%%%%%%%%%%%%%%
%%%%%%%%%%%%%%%%%%%%%%%%%%%%%%%%%%%%%%%%%%%%%%%%%%%%%%%%%%%%%%%%%%%%%%%%%%%%%%%%
%%%%%%%%%%%%%%%%%%%%%%%%%%%%%%%%%%%%%%%%%%%%%%%%%%%%%%%%%%%%%%%%%%%%%%%%%%%%%%%%
\newsection{Localized energy spaces}
% {\color{red} above here... $\Phi$, $\Psi$, and $\Theta$ should be
%   renamed so these can be reserved for the Fourier variables in the
%   below.  Or just abuse the notation... create a new section and start
% by clearing those variables.}
% {\color{red}
% First pass to $\tilde{v}$ coordinates... and $r$ instead of $x$?
% Seems easier to stick with $x$ for now and to do the simple conversion
% at the end.  Maybe we should consider rewriting the previous section
% using the $x$ notation, even... and then perturb off of that.  Thoughts?
% }
%
% {\color{blue}
% I don't think we need to pass to the $\tilde{v}$ coordinates, as the whole analysis takes place far from the event horizon. As opposed to $x$ vs $r$ I switched to $r$
% }
%
We now return to the $(1+4)$-dimensional Myers-Perry space-time with
metric given by \eqref{MPmetric}, and we seek to define localized
energy norms that reflect the more complicated trapping on this
background and to state our main theorem.

We let $\tau, \Xi, \Phi, \Psi,$ and $\Theta$ denote the Fourier
variables for $t, x, \phi, \psi,$ and $\theta$.  We are abusing
notation as we early used some of these to denote constants of
motion.  From this point forward, however, the notation will be
reserved solely for the Fourier variables.

We set
\[p(x,\theta, \tau, \Xi, \Phi, \Psi, \Theta) = 
g^{tt}\tau^2 + 2g^{t\phi}\tau \Phi + 2g^{t\psi}\tau\Psi +
g^{\phi\phi}\Phi^2 + g^{\psi\psi}\Psi^2 + 2g^{\phi\psi}\Phi\Psi +
g^{xx}\Xi^2 + g^{\theta\theta}\Theta^2
\]
to be the Hamiltonian, which vanishes on any null geodesic.  The
Hamilton flow indicates 
\begin{equation}
  \label{dotr}
  \dot{x} = \frac{\partial p}{\partial \Xi} = \frac{8\Delta}{\rho^2}\Xi\\
\end{equation}
\begin{equation}
  \label{dotxi}
  \dot{\Xi} = - \frac{\partial p}{\partial x} = -\Bigl(g^{tt}_{,x}\tau^2 + 2g^{t\phi}_{,x}\tau \Phi + 2g^{t\psi}_{,x}\tau\Psi +
g^{\phi\phi}_{,x}\Phi^2 + g^{\psi\psi}_{,x}\Psi^2 + 2g^{\phi\psi}_{,x}\Phi\Psi +
g^{xx}_{,x}\Xi^2 + g^{\theta\theta}_{,x}\Theta^2\Bigr).
\end{equation}
Equivalent to the latter, we have
\begin{align*}
  \rho^2 \dot{\Xi} &=
  -\frac{\partial}{\partial x} (\rho^2 p) + \frac{\partial}{\partial
    x}(\rho^2) p\\
&=R_{a,b}(x,\tau,\Phi,\Psi) \Delta^{-2} + \frac{\partial}{\partial x}(\rho^2) p - 4(2x+a^2+b^2-r_s^2) \Xi^2.
\end{align*}
To be explicit, we have
 \begin{multline*}
   R_{a,b} = \Bigl(x^2(x+b^2)(x-2r_s^2+b^2) + a^4(x+b^2)^2 +
   a^2(b^2(4x^2-2r_s^2x+r_s^4) + 2x^2(x-r_s^2)+2b^4x)\Bigr)\tau^2
 \\+2ar_s^2(x+b^2(2x-r_s^2)+b^4)\tau\Phi +
 2br_s^2(x^2+a^2(2x-r_s^2)+a^4)\tau \Psi
 \\+(b^2(x^2+2(r_s^2+b^2)x+(r^4-b^4)) - a^2(x^2-2b^2(x-r^2)+b^4))\Phi^2
 \\+(a^6 - a^4 (b^2 + 2 r_s^2 - 2 x) - b^2 x^2 + 
  a^2 (r_s^4 - 2 r_s^2 x + x (-2 b^2 + x))\Psi^2
 \\-2abr_s^2(2x-r_s^2+a^2+b^2)\Phi\Psi.
 \end{multline*}
%
% {\color{blue} Here we insert the proof that this has a simple root $r=r_K(\tau,\Phi,\Psi)$. }
%
% From now on . We now note that
%\[
% r_K(\tau,\Phi,\Psi) = r_{ps} + F_{a,b}(\frac{\Phi}{\tau}, \frac{\Psi}{\tau})
%\] 
% where $F_{a,b} \lesssim \epsilon_0$ when $|\frac{\Phi}{\tau}|, |\frac{\Psi}{\tau}| \lesssim 1$.
 
Along any null geodesic, we have that $p=0$.  Moreover, as we proved
in Section \ref{theMetric}, trapped null geodesics have constant $x$,
and thus, by \eqref{dotr}, $\Xi\equiv 0$.  Hence, corresponding to any
trapped null geodesic, there must be a root $r_{a,b}(\tau,\Phi,\Psi)$
solving
\[R_{a,b}(r_{a,b}^2,\tau,\Phi,\Psi) = 0.\]

Note that $\tau^{-2}R_{a,b}$ is a fourth order polynomial in $x$ and $\tau^{-2}R_{0,0} = x^3(x-2r_s^2)$ has a simple root at $x=r_{ps}^2$. Moreover, near $x=r_{ps}^2$, there is a constant $C_{mp}$ so that $|\Phi|, |\Psi| \leq C_{mp}|\tau|$ since $p=0$. Hence $\tau^{-2}R_{a,b}$ is a small perturbation of $\tau^{-2}R_{0,0}$ for small enough $a,b$, and has a simple root near $r=r_{ps}$.

  It will be more convenient to work with the $r$ coordinate, $x=r^2$, 
 and its corresponding Fourier variable $\xi=2r\Xi$ instead of $x$ and
 $\Xi$.  For future reference, we record that 
 \begin{equation}\label{rderiv} 
  \partial_r (\rho^2 p) = -2r R_{a,b} \Delta^{-2} + \partial_r (\frac{\Delta}{r^2}) \xi^2
 \end{equation}
 \begin{equation}\label{xideriv}
 \partial_{\xi} (\rho^2 p) = 2\frac{\Delta}{r^2} \xi.
  \end{equation}

  We are now ready to define the local energy spaces, in a similar
  fashion to \cite{TT}.
 Naively, one would like to replace the $r-r_{ps}$ factor in \cite{LM} by a quantization of $r-r_{a,b}(\tau,\Phi,\Psi)$. The problem with this approach is that any such quantization depends nontrivially on $\tau$, which makes uniform energy bounds on $\tilde{v}$-slices difficult to prove. Instead, near $r=r_{ps}$ we can write
\[
p = g^{tt} (\tau-\tau_1(r,\theta,\xi,\Psi,\Phi,\Theta)) (\tau-\tau_2(r,\theta,\xi,\Psi,\Phi,\Theta))
\]  
for some real smooth 1-homogeneous symbols $\tau_1$ and $\tau_2$. We define
\[
 c_i(r, \theta, \xi,\Psi,\Phi,\Theta) = \chi_{\geq 1}(r-r_{a,b}(\tau_i, \Phi, \Psi)) 
\]
where $\chi_{\geq 1}$ is a smooth symbol which equals $1$ for frequencies $\gg 1$ and $0$ for frequencies
 $\lesssim 1$. The role of the cutoff is to transform the homogeneous
 symbol into a classical one. Note that since low frequencies are
 already controlled without degeneracy near the trapped set, this
 makes no difference in our estimates. 

%{\color{red} I'm just slightly confused.  Which variable are we
%  cutting off in?  $D$?}

 We use the symbols $c_i$ to define microlocally weighted function spaces in a neighborhood $V\times\S^3$ of $\{r=r_{ps}\}\times \S^3$. We set
\[
\|u\|_{L_{c_i}^2}^2 = \|c_i^w(x,D) u\|_{L^2}^2 + \|u\|_{H^{-1}}^2.
\]  
  For the inhomogeneous term, we define the dual norm
\[
 \|g\|_{c_iL^2}^2 = \inf_{c_i(x,D) g_1 + g_2 = g} (\|g_1\|_{L^2}^2 +
 \|g_2\|_{H^1}^2).
\]  
For the remainder of the paper, we shall be using the $r$ variable
rather than the $x$ variable from Section \ref{theMetric}.  In another
abuse of notation, $x$ is now used to denote $x=r\omega$ analogous to
standard spherical coordinates.

 Pick $\chi(r)$ to be a cutoff function supported in $V$ and equal to $1$ near $r=r_{ps}$. We set our local energy norm to be
\begin{equation}
\begin{split}
\|u\|_{LE_{mp}^1} = & \|\chi(D_t - \tau_2(x,D))\chi u\|_{L^2_{c_1}} 
+ \|\chi (D_t - \tau_1(x,D))\chi u\|_{L^2_{c_2}} + \|(1-\chi^2) \partial_t  u\|_{LE_M}
\\ &\ +  \|(1-\chi^2) \ang  u\|_{LE_M}
+ \| \partial_r u\|_{LE_M} + \| r^{-3/2} u \|_{L^2}.
\end{split}
\label{leK}\end{equation}

 For the inhomogeneous term, we set the dual norm to be
\begin{equation}\label{leK*}
\|f\|_{LE_{mp}^*} = \| (1-\chi) f\|_{LE_M^*} + \|\chi f\|_{c_1 L^2 + c_2 L^2}
\end{equation}

% {\color{blue} We need to define the $LE_M$ norm somewhere, maybe in the Schwarzschild section? Also do we have the improvements at infinity in Schwarzschild? If not, I'll change to whatever we have. }

Letting $\Box_{mp}$ denote the d'Alembertian in the Myers-Perry
metric, the main result in our paper is the following: 
 
\begin{theorem}\label{Kerr}  For $|a|, |b|$ sufficiently small and for
  some 
 $r_- < r_e < r_+$, let $u$ solve $\Box_{mp} u = f$ in $\M_R$. Then
\begin{equation}
 \|u \|_{LE_{mp}^1}^2 + \sup_{\tilde v} E[u](\tilde v) + E[u](\Sigma_R^+)
 \lesssim E[u](\Sigma_R^-)+ \|f \|_{LE_{mp}^*}^2.
\end{equation}
in the sense that the left hand side is finite and the inequality
holds whenever the right hand side is finite.
\end{theorem}

\newsection{Proof of Theorem \ref{Kerr}}
 We now prove the localized energy estimate from the previous section assuming \newline$\max(|a|, |b|) \leq \e_0 \ll 1$. The arguments are fundamentally the same as the ones in \cite{TT}, and some of the results will be cited without proof.
 
  The main idea is to use the multiplier method. Unfortunately, due to the complicated nature of the trapping, no differential operator provides us with a positive local energy norm. Instead we rely on the smoothed out (near the photonsphere) vector field $X$ to control the $LE_{mp}$ norm away from the trapped set, to which we add a pseudodifferential correction near the trapped set. Since we would also like to establish uniform energy bounds, we will pick the correction to be a differential operator of first order in $t$, which will allow us to integrate by parts with respect to time. 
  
 % For a Lorentzian metric $g$, the energy-momentum tensor is given by
% \[
% Q_{\alpha\beta}[u]=\partial_\alpha u \partial_\beta u -
% \frac{1}{2}g_{\alpha\beta}\partial^\gamma u \partial_\gamma u.
% \] 

%  For a $C^1$ vector field $X$, a scalar function $q$ and a 1-form $m$ we define
% \[
% P_\alpha[u,X,q,m] = Q_{\alpha\beta}[u]X^\beta + q \phi \partial_\alpha u
% - \frac12 \partial_\alpha q u^2 + \frac{1}{2}m_{\alpha}u^2
% \] 

%  The divergence formula then gives
% \begin{equation}
% \nabla^\alpha P_\alpha[u,X,q,m] =  \Box_g u \Bigl(Xu +
%  q u\Bigr)+ Q[u,X,q,m],
% \label{div}\end{equation}
% where
% \[
%  Q[u,X,q,m] = 
% \frac{1}{2}Q_{\alpha\beta}[u]\pi^{\alpha\beta} + q
% \partial^\alpha u\, \partial_\alpha u + m_\alpha u\,
% \partial^\alpha u + (\nabla^\alpha m_\alpha -\frac{1}{2}
% \nabla^\alpha \partial_\alpha q) \, u^2.
% \]
% and
% \[
% \pi_{\alpha\beta}=\frac{1}{2}(\nabla_\alpha X_\beta + \nabla_\beta X_\alpha)
% \]
% is the deformation tensor of $X$. 

If one integrates \eqref{divP} on the domain
\[
D = \{ 0 < \tv < \tv_0,\ r > r_e \},
\]
local energy estimates are established as long as the boundary terms satisfy
\begin{equation}
  BDR[u]  \leq c_1 E[u](\Sigma_R^-) - 
c_2 (E[u](\tilde t_0) + E[u](\Sigma_R^+)), \qquad c_1, c_2 > 0
\label{bdrpos}\end{equation}
and $Q[u,X,q,m] \geq 0$. 

From here one, we will use the sub(super)scripts $S$ and $mp$ to
denote the Schwarzschild and the Myers-Perry metrics respectively.  
By choosing $X$, $q$ and $m$ like in Section 3, we know that 
\begin{equation}
Q^S[u,X,q,m] \gtrsim r^{-4} (\partial_r u)^2 + \left(1-\frac{r_{ps}}r
\right)^2 (r^{-4} (\partial_{\tv} u)^2 + r^{-1}|\ang u|^2) + r^{-3}
u^2.
\label{posS}\end{equation}
and the boundary terms satisfy \eqref{bdrpos}

The same computation can be performed for the Kerr metric. Precisely, with
$\partial$ standing for $\partial_t$ and $\partial_x$, $x = r \omega$,
\begin{equation}\label{aproxmet}
 |\partial^\alpha [(g_{mp})_{ij}-(g_{S})_{ij}]|\leq c_\alpha
 \frac{\e_0}{r^{2+|\alpha|}}, \qquad 
 |\partial^\alpha [(g_{mp})^{ij}-(g_{S})^{ij}]|\leq c_\alpha
 \frac{\e_0}{r^{2+|\alpha|}}
\end{equation}
and thus one can easily check that
\begin{equation}
|P_\alpha^S[u,X,q,m] -P_\alpha^{mp}[u,X,q,m]| \lesssim \frac{\e_0}{r^2}
|\partial u|^2, 
\label{pdiff}\end{equation}
respectively 
\begin{equation}
|Q^S[u,X,q,m] -Q^{mp}[u,X,q,m]| \lesssim \e_0 \left( \frac{1}{r^2}
|\partial u|^2 + \frac{1}{r^4} |u|^2\right).
\label{qdiff}\end{equation}
 Thus \eqref{bdrpos} still holds; however, we can only say that
\begin{equation}
Q^{mp}[u,X,q,m] \gtrsim r^{-4} |\partial_r u|^2 + \left[\left(1-\frac{r_{ps}}r
\right)^2-C\e_0\right] (r^{-4} |\partial_{\tv} u|^2 + r^{-1}|\ang u|^2) + r^{-3}
u^2,
\label{qkbd}\end{equation} 
and the right hand side is no longer positive definite near $r = r_{ps}$.

 In order to correct this, we will add a pseudodifferential correction to $X$ and $q$. In order to quantize the symbols, we will use a Weyl calculus with respect to the metric $dV_{mp} = r\rho^2 drd\omega$. We thus slightly abuse the notation and redefine the Weyl quantization as
 \[
  s^w:= \frac{r}{\rho} \ s^w\  \frac{\rho}{r}
 \]
 so that real symbols get quantized to self-adjoint operators with respect to $L^2(dV_{mp})$.
 
  Let 
\[
S =  i s_1^w + s_0^w \partial_t, \qquad E =  e_0^w 
+ \frac{1}i  e_{-1}^w \partial_t
\]
where $s_1 \in S^1$, $s_0,e_0 \in S^0$ and $e_{-1} \in S^{-1}$ are
real symbols with kernels supported close to $r=r_{ps}$. One can now compute
\begin{equation}
 \Re \int_D \Box_{mp} u \cdot
(S+E) u \, dV_{mp} = \int_D Qu \cdot u \, dV_{mp}+ BDR^{mp}[u,S,E]
\label{intdivs}\end{equation}
where
\begin{equation}\label{Qform}
Q = \frac12 ([\Box_{mp}, S] + \Box_{mp} E + E\Box_{mp}) = q_2^w  + 2 q_1^w D_t + q_0^wD_t^2 + q_{-1}^w D_t^3
\end{equation}
with $q_j \in S^j$. 

 Since the Weyl quantization is only with respect to the spatial variables, we define the matrix-valued pseudodifferential operator
\[
 \tilde Q = \left( \begin{array}{cc}
q_2^w & q_1^w \\
q_1^w & q_0^w \end{array} \right).
\] 
Let
\begin{equation}
IQ^{mp}[u,S,E]  = \int_D q_{2}^w u \cdot \overline u + 2\Re
q_1^w u \cdot \overline {D_t u} + q_0^w {D_t u} \overline {D_t u} \
dV_{mp}
\label{IQrep}\end{equation}
and note that
\[
 IQ^{mp}[u,S,E]  = \int_0^{\tv_0} \la \tilde Q \vec{u}, \vec{u} \ra d\tv
\]
where $\vec{u} = (u, D_t u)$.

 We will pick $S$ and $E$ so that $q_{-1}^w = 0$; in this case, after integrating by parts in time, \eqref{intdivs} becomes
\begin{equation}
 \Re \int_D \Box_{mp} u \cdot
(S+E) u dV_{mp} = IQ^{mp}[u,S,E] + BDR^{mp}[u,S,E]. 
\end{equation} 
 The exact form of the boundary terms does not matter, since we will pick the symbols $s_i\in \e_0S^i$ to be small when $\e_0 \ll 1$. In particular, since there are no boundary terms at $r=r_e$, this means that 
\[
|BDR^{mp}[u,S,E]| \lesssim \e_0 (E[u](0) + E[u](\tv_0)).
\] 
 The local energy estimates now follow if we can pick $S$ and $E$ so that 
\begin{equation}
  \int_D  Q^{mp}[u,X,q,m]  dV_{mp}  + IQ^{mp}[u,S,E] 
\gtrsim \|u\|_{LE_{mp}}^2.
\label{mek}\end{equation} 

 Near $r=r_{ps}$, the principal symbol of the quadratic form on the left in \eqref{mek}
is (see \eqref{Qform})
\[
\frac{1}{2i} \{ p, X+s\} + p(q+ e)
\]
where for convenience we denote
\[
s = s_1 + \tau s_0, \qquad e = e_0 + \tau e_{-1}.
\]
 
 In order for \eqref{mek} to hold, the above symbol must dominate the principal part of the $LE_{mp}$ norm, i.e.
\[
 \frac{1}{2i} \{ p, X+s\} + p(q+ e) \gtrsim c_2^2(\tau -\tau_1)^2 + 
c_1^2(\tau-\tau_2)^2 + \xi^2.
\] 
 Unfortunately this is not enough. Indeed, since $\tilde Q$ is a
 matrix valued operator of second order, the Fefferman-Phong
 inequality does not hold. What we do instead is write the left hand
 side as a sum of squares dominating the right hand side. We have: 

\begin{lemma} \label{squaresum}
  Let $\e_0$ be sufficiently small. Then there exist smooth homogeneous
  symbols $s \in \e_0 (S^1_{hom} + \tau S^0_{hom})$, $e \in \e_0 (S^0_{hom} + \tau S^{-1}_{hom})$ so that for $r $ close to $r_{ps}$ we have 
  \begin{equation}
\rho^2\left( \frac1{2i} \{ p, X+s \} + p(q+e)\right) = \sum_{j = 1}^{11} \mu_j^2 \gtrsim c_2^2(\tau -\tau_1)^2 + 
c_1^2(\tau-\tau_2)^2 + \xi^2
\label{sumsqK}\end{equation}
for some $\mu_j \in S^1_{hom}+\tau S^0_{hom}$ that depend smoothly on $a,b$ and in addition satisfy the conditions:

i) $\mu_i$ are differential operators for $a=b=0$, $i=\overline{1, 9}$

ii) $\mu_{10}$ and $\mu_{11}$ are small, in the sense that
\[
(\mu_{10},\mu_{11}) \in \sqrt{\e_0} (S^1_{hom}+\tau S^0_{hom}).
\]
\end{lemma}

\begin{proof}
 We will start by reformulating the results of
 Section \ref{Schwarzschild} in the setting of pseudodifferential operators. We know that near $r=r_{ps}$ the symbol of $X$ is $if(r) (r-r_{ps})\xi$ for some smooth function $f>0$. Then near $r=r_{ps}$, we can write
\[
 Q^S [u,X,q,m] = q^{S, \alpha\beta} \partial_{\alpha}u\partial_{\beta} u + q^{S, 0}u^2
\] 
 where the principal symbols are
\[
 q^S = q^{S, \alpha\beta}\eta_{\alpha}\eta_{\beta} =  \frac{1}{2i} \{p_S,X\} + q p_S, \qquad q^{S,0} = -  \frac12
\nabla^\alpha \partial_\alpha q.
\] 
 The symbol of $\Box_S$ is
\[
p_S = - \left(1-\frac{r_s^2}{r^2} \right)^{-1} \tau^2 + \left(1-\frac{r_s^2}{r^2}
\right) \xi^2 + \frac{1}{r^2} \lambda^2
\]
where $\lambda$ stands for the spherical Fourier variable.
 
 We can now compute
\begin{equation}
\begin{split}
r^2 q^S =&\   \frac{1}{2i} \{r^2 p_S,X\} +( q - r^{-1} f(r)(r-r_{ps})) (r^2 p_S) 
\\ =&\  \alpha_S^2(r)  \tau^2 + \beta_S^2(r)  \xi^2 + \tilde q(r)(r^2 p_S)   
\end{split}
\label{sqss}\end{equation}
where, near  $r = r_{ps}$,
\[
\alpha_S^2(r) =  \frac{ r^3 (r+\sqrt{2} r_s)  f(r) (r-r_{ps})^2}{(r^2-r^2_{s})^2},
\]
\[
\quad \beta_S^2(r) = (r^2 - r_s^2)f  + (r-r_{ps})(f'(r^2 -r_s^2) -rf),
\]
respectively 
\[
\tilde q (r) =  q - r^{-1} f(r)(r-r_{ps}).
\]
The facts that $f > 0$ and $\Bigl[\Bigl(1-\frac{r_s^2}{r^2}\Bigr)^{-1} (r-r_{ps})f(r)\Bigr]'>0$
are used
to write the first two
coefficients as squares.

 Due to the results of Section \ref{Schwarzschild}, we know that 
\begin{equation}
q^S \gtrsim \xi^2 + (r-r_{ps})^2 (\tau^2 + \lambda^2), \qquad q^{S,0}
> 0
\end{equation}
near $r=r_{ps}$, 
which implies that $\tilde q$ is a multiple of $(r-r_{ps})^2$ and moreover we can write
\[
-g^{tt} r^2 \tilde q = \nu(r) \alpha_S^2(r), \qquad 0 < \nu < 1.
\] 
The symbol $\lambda^2$ of the spherical Laplacian can also be written
as sums of squares of differential symbols,
\[
\lambda^2 = \sum_{i=1}^6 \lambda_{i}^2 
\]
where in Euclidean coordinates we can write
\[
\{ \lambda_{i} \} = \{ x_k \eta_j -x_j \eta_k, 1\leq k<j\leq 4 \}.
\]
 This leads to the sum of squares
\begin{equation}
  r^2 q^S =  (1-\nu(r)) \alpha_S^2(r)  \tau^2 + \beta_S^2(r)  \xi^2 +
  \nu(r)\alpha_S^2(r) r^{-2}(\sum_{i=1}^6 \lambda_{i}^2
  + (r^2-r_s^2) \xi^2).
\label{sumsqS}\end{equation}

 The natural counterpart for the Myers-Perry spacetime is the symbol
\[
 \tilde s = i f(r) (r-r_{a,b} (\tau, \Phi, \Psi))\xi.
\] 
 This is a homogeneous symbol that coincides with $X$ in the Schwarzschild case $a=b=0$, and it is well
defined for $r$ near $r_{ps}$ and $|\Phi|, |\Psi| < C_{mp}|\tau|$. In
particular it is well defined in a neighborhood of the characteristic
set $p =0$, which is all we need.

%{\color{blue}  the 4 here needs to match what is given in the previous section?}

 We can now compute the Poisson bracket on the characteristic set $\{p=0\}$: 
\[
\begin{split}
\frac{1}{i}  \{ \rho^2 p, \tilde s\} =&\ -  (\rho^2 p)_r f(r)(r-r_{a,b}(\tau,\Phi, \Psi))+  
\xi (\rho^2 p)_\xi \partial_{r}\left(f(r)(r-r_{a,b}(\tau,\Phi, \Psi))\right) 
\\
 = &\ 2r f(r) R_{a,b}\Delta^{-2}(r-r_{a,b}(\tau,\Phi, \Psi))
\\ &\ +
 \left[2\frac{\Delta}{r^2}\partial_{r}\bigl(f(r)(r-r_{a,b}(\tau,\Phi,
   \Psi))\bigr) - (\partial_r \frac{\Delta}{r^2}) f(r)
   (r-r_{a,b}(\tau,\Phi, \Psi))  \right] \xi^2. 
\end{split}
\] 
 Since $r_{a,b}(\tau,\Phi,\Psi)$ is the unique zero of $R_{a,b}$ near
$r=r_{ps}$ and is close to $r_{ps}$, it follows that we can write
\begin{equation}
\frac{1}{2i} \{ \rho^2 p,\tilde s\} = 
\alpha^2(r,\tau,\Phi,\Psi) \tau^2 (r-r_{a,b}(\tau,\Phi,\Psi))^2 +
\beta^2(r,\tau,\Phi,\Psi)\xi^2  \qquad \text{on} \ \ \{p = 0\}
\label{pbs}\end{equation}
where $\alpha, \beta \in S^0_{hom} $ are positive
symbols.

Unfortunately $\tilde s$ is not a polynomial in $\tau$. To remedy that we first note that 
\[
\tilde s - i f(r) (r-r_{ps}) \xi \in \e_0 S^1_{hom}.
\]
Hence by (the simplest form of)
the Malgrange preparation theorem we can write
\[
\frac{1}i \tilde s = (r-r_{ps})f(r)\xi  +  s_1(r,\xi, \theta,\Theta,\Phi,\Psi)+
s_0(r,\xi, \theta,\Theta,\Phi,\Psi) \tau)+  
\gamma(\tau, r,\xi, \theta,\Theta,\Phi,\Psi) p
\]
with $s_1 \in \e_0 S^1_{hom}$, $s_0 \in \e_0 S^0_{hom}$ and $\gamma \in
\e_0 S^{-1}_{hom}$.  Then we define the desired symbol $s$ by
\[
s = i(s_1 + s_0 \tau).
\]

 We also define 
\[
\alpha_{i} =\frac{2|\tau_{i}|}{\tau_{1} -\tau_{2}}
 \alpha(r,\tau_{i},\Phi,\Psi) (r - r_{a,b}(\tau_{i},\Phi,\Psi)) \in S^0_{hom},
\qquad \beta_{i} = \beta(r,\tau_{i},\Phi,\Psi),
\] 
 and let $C$ be a large constant. Then we can find $e \in \e_0 (S^0_{hom} + \tau S^{-1}_{hom})$ so that \eqref{sumsqK} holds for
\[
\mu_{i}^2 =  \frac{\lambda_{i}^2}{r^2( \lambda^2
  + (r^2-r_s^2) \xi^2} \frac{\nu}4 (\alpha_1 (\tau-\tau_2)
  -\alpha_2(\tau-\tau_1))^2, \qquad i=\overline{1,6},
\] 
\[
 \mu_7^2 = \frac{(r^2-r_s^2) \xi^2}{r^2( \lambda^2
  + (r^2-r_s^2) \xi^2} \frac{\nu}4 (\alpha_1 (\tau-\tau_2)
  -\alpha_2(\tau-\tau_1))^2
\]
\[
\mu_8^2 = \frac{1-\nu}4 (\alpha_1 (\tau-\tau_2)
  +\alpha_2(\tau-\tau_1))^2, \qquad \mu_9^2 = \frac12 (\beta_1^2 + \beta_2^2 - C\e_0) \xi^2
\]
\[
\mu_{10}^2 =  \frac{(C\e_0 - \beta_2^2+\beta_1^2) (\tau-\tau_2)^2} {2(\tau_1-\tau_2)^2}
  \xi^2, \qquad \mu_{11}^2 = \frac{ (C\e_0-\beta_1^2+\beta_2^2)
    (\tau-\tau_1)^2}{2(\tau_1-\tau_2)^2} \xi^2 
\]

 The proof is identical to the one in Lemma 4.3 of \cite{TT}, so we shall skip it.
 
\end{proof}

 To finish the proof, we will pick $S$ and $E$ so that \eqref{mek} holds. Since $s$ and $e$ are homogeneous, we need to truncate away the low frequencies, so we redefine
\[
 s:= \chi_{>1} s, \qquad e:= \chi_{>1} e
\] 
 where $\chi$ is a smooth symbol equal to $1$ for frequencies greater than $1$ and $0$ for frequencies $\ll 1$. We also need to truncate near the trapped set. Let $\chi$ be a smooth cutoff function supported near $r_{ps}$
which equals $1$ in a neighborhood of $r_{ps}$, chosen so that we have a
smooth partition of unity in $r$,
\[
1 = \chi^2(r) + \chi_{o}^2(r).
\]
We now define the operators 
\[
S = \chi s^w \chi, \qquad E = \chi e^w \chi - e^w_{aux} D_t
\]
 where the operator $e^w_{aux}$ is chosen so that the $D_t^3$ term in \eqref{Qform} vanishes:
\[
g^{tt} e^w_{aux} +  e^w_{aux}g^{tt} =  q_{-1}^w 
\]

 The proof that \eqref{mek} holds for small enough $\e_0$ is identical to the one in \cite{TT}. 
%%%%%%%%%

\bigskip
%%%%%%%%%%%%%%%%%%%%%%%%%%%%%%%%%%%%%%%%%%%%%%%%%%%%%%%%%%%%%%%%%%%%%%%%%%%%%%%%%%%%%%%
%%%%%%%%%%%%%%%%%%%%%%%%%%%%%%%%%%%%%%%%%%%%%%%%%%%%%%%%%%%%%%%%%%%%%%%%%%%%%%%%%%%%%%%
%%%%%%%%%%%%%%%%%%%%%%%%%%%%%%%%%%%%%%%%%%%%%%%%%%%%%%%%%%%%%%%%%%%%%%%%%%%%%%%%%%%%%%%

\end{document}